\documentclass[10pt,authoryear]{article}%
\usepackage{amsmath}
\usepackage{amsfonts}
\usepackage{mathrsfs}
\usepackage{amssymb,bbm, color}
\usepackage{enumerate}
\usepackage[hidelinks,colorlinks=true,linkcolor=blue,citecolor=blue]{hyperref}
\usepackage{graphicx}
\numberwithin{equation}{section}
\usepackage[body={15.5cm,21cm}, top=3cm]{geometry}%
 \usepackage{natbib}
\providecommand{\U}[1]{\protect\rule{.1in}{.1in}}
\providecommand{\U}[1]{\protect \rule{.1in}{.1in}}
\newtheorem{theorem}{Theorem}[section]

\newtheorem{corollary}[theorem]{Corollary}

\newtheorem{definition}[theorem]{Definition}
\newtheorem{example}[theorem]{Example}

\newtheorem{lemma}[theorem]{Lemma}

\newtheorem{proposition}[theorem]{Proposition}
\newtheorem{remark}[theorem]{Remark}

\newenvironment{proof}[1][Proof]{\noindent \textbf{#1.} }{\  \rule{0.5em}{0.5em}}

\def \P{\mathsf{P}}
\def \E{\mathsf{E}}
\def \hE {\widehat{\mathbb E}}
\begin{document}
	\title{Stochastic Differential Equations Driven by $G$-Brownian Motion with Mean Reflections}
	\author{Hanwu Li\thanks{Research Center for Mathematics and Interdisciplinary Sciences, Shandong University, Qingdao 266237, Shandong, China. Email: lihanwu@sdu.edu.cn.}
	\thanks{Frontiers Science Center for Nonlinear Expectations (Ministry of Education), Shandong University, Qingdao 266237, Shandong, China.}
	\and Ning Ning\thanks{Department of Statistics,
		Texas A\&M University, College Station, Texas, USA. Email: patning@tamu.edu.}}
	\date{}
	\maketitle
	\begin{abstract}
	In this paper, we study the mean reflected  stochastic differential equations driven by $G$-Brownian motion, where the constraint depends on the expectation of the solution rather than on its paths. Well-posedness is achieved by first investigating the Skorokhod problem with mean reflection under $G$-expectation. Two approaches to constructing the solution are introduced, both offering insights into desired properties and aiding in the application of the contraction mapping method. 
	\end{abstract}
	
	
	\textbf{MSC-classification}: 60G65, 60H10
	
	
\section{Introduction}
We firstly give the background in Subsection \ref{sec:Background} and then state our contributions in Subsection \ref{sec:contributions}, followed with the organization of the paper in Subsection \ref{sec:organization}.

\subsection{Background} 
\label{sec:Background}	
Stochastic differential equations (SDEs) with reflecting boundaries, commonly referred to as reflected SDEs, were introduced by Skorokhod in the 1960s \citep{Skorokhod1}. 
Later, \cite{el1997reflected} introduced the reflected backward SDE (BSDE), where the first component of the solution is constrained to remain above a specified continuous process, known as the obstacle. Reflected SDEs and reflected BSDEs are intimately linked to various fields including optimal stopping problems (see, e.g., \cite{cheng2013optimal}), pricing for American options (see, e.g., \cite{el1997reflected2}), and the obstacle problem for partial differential equations (PDEs) (see, e.g., \cite{bally2002reflected}). Hence, they have attracted a great deal of attention in the probability community, such as \cite{CE,T,LS,ma2005representations, BKR,hamadene2010switching, ning2021well, ning2023multi, NING2024One} and the references therein, providing a comprehensive overview of this theory. In all the aforementioned papers, the constraints depend on the paths of the solution.

Over the past decade, \cite{Bouchard2015BSDEs} pioneered the modeling of BSDEs with mean reflection, where the terminal condition constrains the distribution of the BSDE at terminal time. Mean reflected BSDEs (MRBSDEs) were formally introduced by \cite{briand2018bsdes}. Subsequently, the forward version was proposed in \cite{Briand2020Particles}, considering the following type of mean reflected SDEs (MRSDEs): for $t\in[0,T]$,
	 \begin{equation}\label{eq2}
		\begin{cases}
		X_t=x_0+\int_0^t b(X_s)ds+\int_0^t \sigma(X_s)dW_s+A_t, \vspace{0.2cm}\\
               \E[h(X_t)]\geq 0\quad\text{and}\quad
                \int_0^T \E[h(X_t)]dA_t=0,
                \end{cases}
		\end{equation}
where $b,\sigma,h:\mathbb{R}\rightarrow \mathbb{R}$ are given Lipschitz functions and $W$ is a standard Brownian motion defined on a complete probability space $(\Omega,\mathcal{F},\P)$. 
Here, the compensating reflection component depends on the distribution of the solution. The authors established both the existence and uniqueness of the solution. MRSDEs and MRBSDEs have attracted warm interests in the probability community, which include, but are not limited to, the following: quadratic MRBSDEs \citep{hibon2017quadratic}, MRSDEs with jumps \citep{briand2020mean}, large deviation principle for the MRSDEs with jumps \citep{li2018large}, MRSDEs with two constraints \citep{falkowski2021mean}, multi-dimensional MRBSDEs \citep{qu2023multi}, and the well-posedness of MRBSDEs with different reflection restrictions \citep{falkowski2022backward,li2024backward}.

All the above mean reflected problems were considered in the classical probability space until  \cite{LiuW} and \cite{heli} studied the mean reflected BSDE driven by $G$-Brownian motion ($G$-BSDEs).
The introduction of $G$-Brownian motion and $G$-expectation was a significant development in the field of stochastic analysis \citep{P07a,P08a,P19}. The nonlinear G-expectation theory was motivated by the consideration of Knightian uncertainty, especially volatility uncertainty, and the stochastic interpretation of fully nonlinear PDEs. 
Roughly speaking, $G$-expectation can be seen as an upper expectation taking over a non-dominated family of probability measures. Under this framework,  $G$-Brownian motion and the associated $G$-It\^{o}'s calculus were established. \cite{G} obtained the well-posedness of SDEs driven by $G$-Brownian motion ($G$-SDEs). For the reflected case,  \cite{Lin} first investigated the scalar-valued $G$-SDE whose solution is required to be above a prescribed $G$-It\^{o}'s process, and later \cite{LSH} considered the reflected $G$-SDEs in non-convex domains. For  reflected $G$-BSDEs, \cite{li2018reflected}, \cite{li2017reflected} and \cite{li2021backward} tackled the lower obstacle case, the upper obstacle case and the double obstacles case, respectively. Recently, \cite{li2024doubly} established the connection of doubly reflected $G$-BSDEs to fully nonlinear PDEs with double obstacles. Notably, \cite{SWW} studied $G$-SDE whose coefficients may depend on the distribution of the solution, but without considering reflection.

\subsection{Our contributions}
\label{sec:contributions}
				
		In this paper, we introduce $G$-SDE with mean reflection in the following form: for $t\in[0,T]$,
		\begin{equation}\label{MRGSDE}
		\begin{cases}
		X_t=x_0+\int_0^t b(s,X_s)ds+\int_0^t h(s,X_s)d\langle B\rangle_s+\int_0^t \sigma(s,X_s)dB_s+A_t, \vspace{0.2cm}\\
               \hE[l(t,X_t)]\geq 0\quad\text{and}\quad
                \int_0^T \hE[l(t,X_t)]dA_t=0,
                \end{cases}
		\end{equation}
where $B$ represents $G$-Brownian motion, $\hE$ denotes $G$-expectation, and the functions $b$, $h$, $\sigma$, and $l$ satisfy certain regularity conditions, which will be elucidated later along with a review of $G$-expectation theory. The solution to $G$-SDE \eqref{MRGSDE} is a pair of processes $(X,A)$, where $A$ is a nondecreasing deterministic function, behaving in a minimal way such that the Skorokhod condition is satisfied. The well-posedness of a unique solution to $G$-SDE \eqref{MRGSDE} is established in Theorem \ref{main}, with the 
assumptions enforced on the coefficient functions being comparable to those of the non-reflected case studied by \cite{G}.

Theorem \ref{main} is achieved by first investigating the Skorokhod problem with mean reflection under $G$-expectation, which is rigorously defined in Definition \ref{SPwithMR} and practically illustrated in Example \ref{example1}.  A crucial result is the existence and uniqueness of the solution to that problem in Theorem \ref{SP}. Two approaches to constructing the solution are presented. The first approach relies on an intermediate result from \cite{LiuW}, outlined in Proposition \ref{prop3.3}, although further stochastic analysis is evidently required in this context. The second approach requires an additional assumption to establish Theorem \ref{SP} thus denoted as Theorem \ref{SP2}. However, this method establishes a link between the Skorokhod problem with mean reflection and the deterministic Skorokhod problem, leading to the acquisition of desired properties as illustrated in Proposition \ref{proposition2.4} and subsequently in Corollary \ref{corollary2.5}. By combining Theorem \ref{SP}, both construction methods, and those attained properties, we establish the existence and uniqueness of solutions for \eqref{MRGSDE}  through a contraction mapping argument. Following this, the desired moment estimate of the solution to mean reflected $G$-SDEs is attained in Proposition \ref{prop2.6}, ensuring the continuity of the compensating term $A$. However, a more refined result is needed to strengthen the analysis. Specifically, it is desired to show that the compensating term exhibits Lipschitz continuity when the loss function $l$ is sufficiently smooth, a property that is verified in Proposition  \ref{prop2.7}. 

\subsection{Organization of the paper}
\label{sec:organization}
The paper is structured as follows: Section \ref{sec:Preliminaries} provides a review of fundamental notations and results in $G$-framework. The mean reflected Skorokhod problem is defined and examined in Section \ref{sec:Skorokhod}. Section \ref{sec:MRGSDEs} focuses on establishing the well-posedness of the mean reflected $G$-SDE \eqref{MRGSDE}. Throughout the paper, the letter $C$, with or without subscripts, will denote a positive constant whose value may change for different usage.

\section{Preliminaries}
\label{sec:Preliminaries}
We review some fundamental  notions and results of $G$-expectation and $G$-stochastic calculus.  The readers may refer to  \cite{P07a,P08a,P19} for more details. For simplicity, we only consider the one-dimensional $G$-Brownian motion, noting that the results still hold for the multidimensional case. 

Let $\Omega_T=C_{0}([0,T];\mathbb{R})$, the space of
real-valued continuous functions starting from the origin, i.e., $\omega_0=0$ for any $\omega\in \Omega_T$, be endowed
with the supremum norm. Let $\mathcal{B}(\Omega_T)$ be the Borel set and $B$ be the canonical process. Set
\[
L_{ip} (\Omega_T)=\Big\{ \varphi(B_{t_{1}},...,B_{t_{n}}):  \ n\in\mathbb {N}, \ t_{1}
,\cdots, t_{n}\in\lbrack0,T], \ \varphi\in C_{b,Lip}(\mathbb{R}^{ n})\Big\},
\]
where $C_{b,Lip}(\mathbb{R}^{ n})$ denotes the set of all bounded Lipschitz functions on $\mathbb{R}^{n}$.
We fix a sublinear and monotone function $G:\mathbb{R}\rightarrow\mathbb{R}$ defined by
\begin{align}\label{GG}
	G(a)=\frac{1}{2}(\overline{\sigma}^2a^+-\underline{\sigma}^2a^-), 
\end{align}
where $0< \underline{\sigma}^2<\overline{\sigma}^2$. The associated $G$-expectation on $(\Omega_T, L_{ip}(\Omega_T))$ can be constructed in the following way. Given that $\xi\in L_{ip}(\Omega_T)$ can be represented as
$$\xi=\varphi(B_{{t_1}}, B_{t_2},\cdots,B_{t_n}),$$
set for $t\in[t_{k-1},t_k)$ with $k=1,\cdots,n$, 
\begin{displaymath}
	\widehat{\mathbb{E}}_{t}[\varphi(B_{{t_1}}, B_{t_2},\cdots,B_{t_n})]=u_k(t, B_t;B_{t_1},\cdots,B_{t_{k-1}}),
\end{displaymath}
where $u_k(t,x;x_1,\cdots,x_{k-1})$ is a function of $(t,x)$ parameterized by $(x_1,\cdots,x_{k-1})$ such that it solves the following fully nonlinear PDE defined on $[t_{k-1},t_k)\times\mathbb{R}$:
\begin{displaymath}
	\partial_t u_k+G(\partial_x^2 u_k)=0,
\end{displaymath}
whose terminal conditions are given by
\begin{displaymath}
	\begin{cases}
		u_k(t_k,x;x_1,\cdots,x_{k-1})=u_{k+1}(t_k,x;x_1,\cdots,x_{k-1},x), \qquad k<n,\\
		u_n(t_n,x;x_1,\cdots,x_{n-1})=\varphi(x_1,\cdots,x_{n-1},x).
	\end{cases}
\end{displaymath}
Hence, $G$-expectation of $\xi$ is $\widehat{\mathbb{E}}_0[\xi]$, denoted as $\widehat{\mathbb{E}}[\xi]$ for simplicity. The triple $(\Omega_T, L_{ip}(\Omega_T),\widehat{\mathbb{E}})$ is called $G$-expectation space and the process $B$ is called $G$-Brownian motion. 

For $\xi\in L_{ip}(\Omega_T)$ and $p\geq1$, we define
$$\Vert\xi\Vert_{L_{G}^{p}}=(\widehat{\mathbb{E}}|\xi|^{p}])^{1/p}.$$
The completion of $L_{ip} (\Omega_T)$ under this norm  is denote by $L_{G}^{p}(\Omega_T)$.   For all $t\in[0,T]$, $\widehat{\mathbb{E}}_t[\cdot]$ is a continuous mapping on $L_{ip}(\Omega_T)$ with respect to the norm $\|\cdot\|_{L_G^1}$. Hence, the conditional $G$-expectation $\mathbb{\widehat{E}}_{t}[\cdot]$ can be
extended continuously to the completion $L_{G}^{1}(\Omega_T)$. Furthermore, \cite{DHP11} proved that $G$-expectation has the following representation.
\begin{theorem}[\cite{DHP11}]
	\label{the1.1}  There exists a weakly compact set
	$\mathcal{P}$ of probability
	measures on $(\Omega_T,\mathcal{B}(\Omega_T))$, such that
	\[
	\widehat{\mathbb{E}}[\xi]=\sup_{\P\in\mathcal{P}}\E^{\P}[\xi], \qquad\forall \xi\in  {L}_{G}^{1}{(\Omega_T)}.
	\]
	We call $\mathcal{P}$ a set that represents $\widehat{\mathbb{E}}$.
\end{theorem}

For $\mathcal{P}$ being a weakly compact set that represents $\widehat{\mathbb{E}}$, we define the following two Choquet capacities:
\[
V(A)=\sup_{\P\in\mathcal{P}}\P(A) \quad\text{and} \quad v(A)=\inf_{\P\in\mathcal{P}}\P(A), \qquad \forall A\in\mathcal{B}(\Omega_T).
\]
A set $A\in\mathcal{B}(\Omega_T)$ is called polar if $V(A)=0$.  A
property holds $``quasi$-$surely"$ (q.s.) if it holds outside a
polar set. In this paper, we do not distinguish two random variables $X$ and $Y$ if $X=Y$, q.s.. The following proposition can be seen as the strict comparison property for $G$-expectation.

\begin{proposition}[\cite{LL}]\label{comparison}
Let $X,Y\in L_G^1(\Omega_T)$ with $X\leq Y$, q.s.. The following properties hold:
\begin{itemize}
\item[(i)] If $v(X<Y)>0$, then $\hE[X]<\hE[Y]$;
\item[(ii)] If $\hE[X]<\hE[Y]$, then $V(X<Y)>0$.
\end{itemize}
\end{proposition}

The lemma below will be utilized in constructing the solution in the next section.
\begin{lemma}[\cite{LiuW}]\label{lemma3.5}
Suppose that $X\in L_G^p(\Omega_T)$ with some $p\geq 1$. Then, for any $\varepsilon>0$, there exists a constant $\delta>0$ such that for all set $O\in \mathcal{B}(\Omega_T)$ with $V(O)\leq \delta$, we have
\begin{align*}
\sup_{t\in[0,T]}\hE\Big[\big|\hE_t[X]\big|^p \mathbbm{1}_O\Big]\leq \varepsilon.
\end{align*}
\end{lemma}

The following result can be regarded as the monotone convergence theorem under $G$-expectation.
\begin{lemma}[\cite{DHP11}]\label{lemma2.5}
Suppose $\{X_n\}_{n\in\mathbb{N}}$ and $X$ are $\mathcal{B}(\Omega_T)$-measurable.
\begin{itemize}
\item[(1)] If $X_n\uparrow X$ q.s. and $\E^\P[X_1^-]<\infty$ for all $\P\in\mathcal{P}$, then $\hE[X_n]\uparrow \hE[X]$.
\item[(2)] If $\{X_n\}_{n\in\mathbb{N}}\subset L_G^1(\Omega_T)$ satisfies $X_n\downarrow X$ q.s., then $\hE[X_n]\downarrow \hE[X]$.
\end{itemize}
\end{lemma} 

We need the following norms and spaces to specify the regularity conditions imposed on the parameter functions.
\begin{definition}
	\label{def2.6} Let $M_{G}^{0}(0,T)$ be the collection of processes such that
	\[
	\eta_{t}(\omega)=\sum_{j=0}^{N-1}\xi_{j}(\omega)\mathbbm{1}_{[t_{j},t_{j+1})}(t),
	\]
	where $\xi_{i}\in L_{ip}(\Omega_{t_{i}})$ for a given partition $\{t_{0},\cdot\cdot\cdot,t_{N}\}$ of $[0,T]$. For each
	$p\geq1$ and $\eta\in M_G^0(0,T)$, denote $$\|\eta\|_{H_G^p}=\left\{\widehat{\mathbb{E}}\bigg(\int_0^T|\eta_s|^2ds\bigg)^{p/2}\right\}^{1/p}\quad\text{and}\quad \Vert\eta\Vert_{M_{G}^{p}}=\left\{\widehat{\mathbb{E}}\bigg(\int_{0}^{T}|\eta_{s}|^{p}ds\bigg)\right\}^{1/p}.$$ 
	Let $H_G^p(0,T)$ and $M_{G}^{p}(0,T)$ be the completions
	of $M_{G}^{0}(0,T)$ under the norms $\|\cdot\|_{H_G^p}$ and $\|\cdot\|_{M_G^p}$, respectively.
\end{definition}

Denote by $\langle B\rangle$ the quadratic variation process of $G$-Brownian motion $B$. For two processes $ \xi\in M_{G}^{1}(0,T)$ and $ \eta\in M_{G}^{2}(0,T)$,
$G$-It\^{o} integrals $(\int^{t}_0\xi_sd\langle
B\rangle_s)_{0\leq t\leq T}$ and $(\int^{t}_0\eta_sdB_s)_{0\leq t\leq T}$ are well defined, see  \cite{lp} and \cite{P19}. The subsequent proposition can be interpreted as the Burkholder--Davis--Gundy (BDG) inequality within $G$-expectation framework.
\begin{proposition}[\cite{P19}]\label{BDG}
	If $\eta\in H_G^{\alpha}(0,T)$ with $\alpha\geq 1$ and $p\in(0,\alpha]$, then we have
	\begin{displaymath}
		\underline{\sigma}^p c\widehat{\mathbb{E}}_t\bigg(\int_t^T |\eta_s|^2ds\bigg)^{p/2}\leq
		\widehat{\mathbb{E}}_t\bigg[\sup_{u\in[t,T]}\bigg|\int_t^u\eta_s dB_s\bigg|^p\bigg]\leq
		\bar{\sigma}^p C\widehat{\mathbb{E}}_t\bigg(\int_t^T |\eta_s|^2ds\bigg)^{p/2},
	\end{displaymath}
	where $0<c<C<\infty$ are constants depending on $p$ and $T$.
\end{proposition}

Let $$S_G^0(0,T)=\Big\{h(t,B_{t_1\wedge t}, \ldots,B_{t_n\wedge t}):t_1,\ldots,t_n\in[0,T],\; h\in C_{b,Lip}(\mathbb{R}^{n+1})\Big\}.$$ For $p\geq 1$ and $\eta\in S_G^0(0,T)$, set $$\|\eta\|_{S_G^p}=\Bigg\{\widehat{\mathbb{E}}\sup_{t\in[0,T]}|\eta_t|^p\Bigg\}^{1/p}.$$ Denote by $S_G^p(0,T)$ the completion of $S_G^0(0,T)$ under the norm $\|\cdot\|_{S_G^p}$. \cite{LPS} proved the following uniform continuity property for the processes in $S_G^p(0,T)$. 
\begin{proposition}[\cite{LPS}]\label{the3.7}
	For $Y\in S_G^p(0,T)$ with $p\geq 1$, we have, by setting $Y_s=Y_T$ for $s>T$,
\begin{align*}
	\limsup_{\varepsilon \rightarrow 0} \hE\left[\sup_{t\in[0,T]}\sup_{s\in[t,t+\varepsilon]}|Y_t-Y_s|^p\right]=0.
\end{align*}
\end{proposition}


\section{The Skorokhod problem with mean reflection}
\label{sec:Skorokhod}
	In this section, we study the Skorokhod problem with mean reflection under $G$-expectation. In Subsection \ref{sec:definition_example}, we rigorously define this problem in Definition \ref{SPwithMR}, present our primary result in Theorem \ref{SP}, and then illustrate with a concrete example in financial mathematics. In Subsection \ref{sec:first}, we give the proof of Theorem \ref{SP}. An alternative way to construct the solution is provided in Subsection \ref{sec:second}. However, this approach necessitates an additional assumption ($H'_l$) to attain Theorem \ref{SP}, thus designated as Theorem \ref{SP2}. The benefit of this alternative lies in its establishment of a connection between the Skorokhod problem with mean reflection and the deterministic Skorokhod problem. Both methods of construction yield intermediary results vital for establishing the well-posedness of the mean reflected $G$-SDE \eqref{MRGSDE} in Section \ref{sec:MRGSDEs}.

\subsection{Definition and an illustration}
\label{sec:definition_example}	
	The assumptions below encapsulate the properties of the running loss function $l$ and the original process 
$S$ under consideration. 
\begin{itemize}
	\item[($H_l$)] The function $l: \Omega_T \times [0,T] \times \mathbb{R} \rightarrow \mathbb{R}$ satisfies the following conditions:
	\begin{enumerate}[{(1)}]
		\item $l(t,x)$ is uniformly continuous with respect to $t$ and $x$, uniformly in $\omega$.
		\item For any $t \in [0,T]$, $l(t,x)$ is strictly increasing in $x$, q.s..
		\item For any $(t,x) \in [0,T] \times \mathbb{R}$, $l(t,x)\in L_G^1(\Omega_T)$  and $\widehat{\mathbb{E}}[\lim_{x\uparrow \infty}l(t,x)]>0$.
		\item For any $(t,x) \in [0,T] \times \mathbb{R}$, $|l(t,x)|\leq \kappa(1+|x|)$ for some $\kappa>0$, q.s..
	\end{enumerate} 
	\item[($H_S$)] There exists some $p\geq1$ such that $S\in S_G^p(0,T)$ and $\widehat{\mathbb{E}}[l(0,S_0)]\geq 0$.
\end{itemize}
For  $C[0,T]$ being the set of all real-valued deterministic continuous functions on $[0,T]$,  define $I[0,T]$ as the subset of $C[0,T]$ consisting of non-decreasing functions with initial value $0$. We now proceed to provide the definition of the solution.
	\begin{definition}\label{SPwithMR}
Considering $(l,S)$ satisfying ($H_l$) and ($H_S$), we define a pair of processes $(X,A) \in S_G^p(0,T) \times I[0,T]$ as a solution to the Skorokhod problem with mean reflection associated with $(l,S)$, denoted as $\mathbb{SP}(l,S)$, if for $t\in[0,T]$,
    \begin{itemize}
    	\item[(a)] $X_t=S_t+A_t$, 
    	\item[(b)] $\widehat{\mathbb{E}}[l(t,X_t)]\geq 0$,
    	\item[(c)] $\int_0^T \widehat{\mathbb{E}}[l(t,X_t)]dA_t=0$.
    \end{itemize} 
	\end{definition}
Now, we present the main result of this section whose proof is provided in Subsection \ref{sec:first}.
   \begin{theorem}\label{SP}
   	Under Assumptions ($H_l$) and ($H_S$), there exists a unique solution $(X,A)\in S_G^p(0,T)\times I[0,T]$ to the Skorokhod problem $\mathbb{SP}(l,S)$.
   \end{theorem}
We illustrate the solution to the Skorokhod problem $\mathbb{SP}(l,S)$ with a concrete example below.
\begin{example}
	\label{example1}
	Let $l$ be a function satisfying $(H_l)$.  For a fixed $t\in[0,T]$, we define a map $\rho_t:L_G^1(\Omega_t)\rightarrow \mathbb{R}$ as 
	\begin{align*}
		\rho_t(X)=\inf\Big\{x\in\mathbb{R}:\widehat{\mathbb{E}}[l(t,x+X)]\geq 0\Big\}.
	\end{align*}
	It is easy to check that $\rho_t$ is nonincreasing and translation invariant. That is 
	\begin{itemize}
		\item If $X,Y\in L_G^1(\Omega_t)$ with $X\leq Y$, then $\rho_t(X)\geq \rho_t(Y)$;
		\item $\rho_t(X+m)=\rho_t(X)-m$, for $X\in L_G^1(\Omega_t)$ and $m\in\mathbb{R}$.
	\end{itemize}
	Therefore, $\rho_t$ can be regarded as a static risk measure. The risk position $X$ is called acceptable at time $t$ if $\rho_t(X)\leq 0$.  In fact, suppose that $\rho_t(X)> 0$, the value $\rho_t(X)$ can be regarded as the amount of money to be added by an agent in order to make the risk position $X$ acceptable at time $t$.  The readers may refer to \cite{ADEH} for the background of risk measures.
	
	Consider an agent  who wants to hold a stock evolving according to 
	\begin{align*}
		S_t=S_0+ \int_0^t \mu (S_s) ds+\int_0^t \sigma(S_s) dB_s,
	\end{align*}
	where $\mu,\sigma:\mathbb{R}\rightarrow\mathbb{R}$ are Lipschitz functions. 
	Given the dynamic risk measure $\{\rho_t\}_{t\in[0,T]}$, one can ask how to make sure that the risk position $S_t$ remains acceptable at each time $t$. To meet this constraint, the agent needs to inject additional cash. We denote $A_t$ as the cumulative amount of cash required to be injected and $X_t$ as the associated value process. Then, we have
	\begin{align*}
		X_t=S_0+ \int_0^t \mu (S_s) ds+\int_0^t \sigma(S_s) dB_s+A_t.
	\end{align*}
	Clearly, the agent would like to manage the risk in a minimal way, which leads to the condition
	\begin{align*}
		\int_0^T \widehat{\mathbb{E}}[l(t,X_t)]dA_t=0.
	\end{align*}
	That is, $(X,A)$ is the solution of a Skorokhod problem  $\mathbb{SP}(l,S)$.
\end{example}

\subsection{First construction of solutions}
\label{sec:first}
 The first method to establish the existence of the Skorokhod problem $\mathbb{SP}(l,S)$ relies on the following two propositions.
   \begin{proposition}[\cite{LiuW}]\label{prop3.3}
   	Let ($H_l$) hold and $X\in L_G^1(\Omega_T)$. Then
   	\begin{itemize}
   		\item[(i)] for each $(t,x)\in[0,T]\times \mathbb{R}$, $l(t,x+X)\in L_G^1(\Omega_T)$,
   		\item[(ii)] the map $x\rightarrow l(t,x+X)$ is continuous under the norm $\|\cdot\|_{L_G^1}$; in particular, $x\rightarrow \widehat{\mathbb{E}}[l(t,x+X)]$ is continuous and strictly increasing.
   	\end{itemize}
   \end{proposition} 

\begin{proposition}\label{prop3.3'}
		Let ($H_l$) hold and  $S\in S_G^p(0,T)$ where $p\geq1$. Then the map $t\rightarrow l(t,S_t)$ is continuous  under the norm $\|\cdot\|_{L_G^1}$; in particular, $t\rightarrow \widehat{\mathbb{E}}[l(t,S_t)]$ is continuous.
\end{proposition}

\begin{proof}
	By Assumption ($H_l$), for any $\varepsilon>0$, there exists a constant $\delta>0$, such that $|l(t,x)-l(s,y)|\leq \varepsilon$ for any $|t-s|+|x-y|\leq \delta$. It is easy to check that, for any $|s-t|\leq \delta$,
	\begin{align*}
		\widehat{\mathbb{E}}|l(t,S_t)-l(s,S_s)|
		\leq &\widehat{\mathbb{E}}|l(t,S_t)-l(s,S_t)|+\widehat{\mathbb{E}}\Big[|l(s,S_t)-l(s,S_s)| \mathbbm{1}_{\{|S_t-S_s|>\delta\}}\Big]\\
		&+\widehat{\mathbb{E}}\Big[|l(s,S_t)-l(s,S_s)| \mathbbm{1}_{\{|S_t-S_s|\leq\delta\}}\Big]\\
		\leq &2\varepsilon+C\widehat{\mathbb{E}}\Big[ \Big(1+\sup_{t\in[0,T]}|S_t|\Big) \mathbbm{1}_{\{|S_t-S_s|>\delta\}}\Big].
	\end{align*}
 By Proposition \ref{the3.7} and  Markov's inequality, we have 
$$\lim_{s\rightarrow t} V(\{|S_t-S_s|>\delta\})=0.$$
Then by Lemma \ref{lemma3.5}, 
$$
\limsup_{s\rightarrow t}\widehat{\mathbb{E}}|l(t,S_t)-l(s,S_s)|\leq 2\varepsilon.
$$
Since $\varepsilon$ can be chosen arbitrarily small, the proof is complete.
\end{proof}\vspace{0.1in}


\begin{proof}[Proof of Theorem \ref{SP}] The proof proceeds in two steps, where we prove existence in the first step and then uniqueness in the second step.
	\bigskip
	
	\noindent{\bf Step 1.} 
	   By Proposition \ref{prop3.3}, $\widehat{\mathbb{E}}[l(t,x+X)]$ is well-defined for $X\in L_G^1(\Omega_T)$.  In order to solve the Skorokhod problem  $\mathbb{SP}(l,S)$, we need to use the operator $L_t:L_G^1(\Omega_T)\rightarrow [0,\infty)$ defined as follows:
	\begin{align}\label{construction}
		L_t(X)=\inf\Big\{x\geq 0:\widehat{\mathbb{E}}[l(t,x+X)]\geq 0\Big\}.
	\end{align}
	Under Assumption ($H_l$), the operator $L_t$ is well-defined since by Lemma \ref{lemma2.5},
	\begin{align*}
		\lim_{x\rightarrow \infty}\hE[l(t,x+X)]=\hE\lim_{x\rightarrow \infty}l(t,x+X)=\hE\lim_{x\rightarrow \infty}l(t,x)>0.
	\end{align*}
		It follows from Propositions \ref{prop3.3} and \ref{prop3.3'} that the map $x\rightarrow \widehat{\mathbb{E}}[l(t,x+S_t)]$ is continuous and strictly increasing, and the map $t\rightarrow \widehat{\mathbb{E}}[l(t,x+S_t)]$ is continuous. We first prove that the map $t\rightarrow L_t(S_t)$ is continuous. First, suppose that $\hE[l(t,S_t)]>0$, which by the definition of $L_t(S_t)$ yields that $L_t(S_t)=0$. Note that 
	\begin{align*}
	\lim_{s\rightarrow t}\hE[l(s,X_s)]=\hE[l(t,X_t)]>0.
	\end{align*}
	 Then, if $|s-t|$ is small enough, we have $\hE[l(s,X_s)]>0$ and consequently, $L_s(X_s)=0$. Second, suppose that $\hE[l(t,S_t)]\leq 0$. For any $\varepsilon>0$, we have
	\begin{align*}
	&\lim_{s\rightarrow t}\hE[l(s,L_t(S_t)-\varepsilon+S_s)]=\hE[l(t,L_t(S_t)-\varepsilon+S_t)]<\hE[l(t,L_t(S_t)+S_t)]=0\\
&\qquad\text{and}\quad	0<\hE[l(t,L_t(S_t)+\varepsilon+S_t)]=\lim_{s\rightarrow t}\hE[l(s,L_t(S_t)+\varepsilon+S_s)],
	\end{align*}
	where we used Proposition \ref{comparison}. Then, if $|s-t|$ is small enough, we have $$\hE[l(s,L_t(S_t)-\varepsilon+S_s)]<0<\hE[l(s,L_t(S_t)+\varepsilon+S_s)],$$ which implies that $|L_s(S_s)-L_t(S_t)|\leq \varepsilon$. Therefore, the map $t\rightarrow L_t(S_t)$ is continuous. Define the function $A$ by setting
	\begin{align*}
		A_t=\sup_{s\in[0,t]}L_s(S_s),
	\end{align*}
    and then define
    \begin{align*}
    	X_t=S_t+A_t.
    \end{align*}
We are going to show that $(X,A)$ is the solution to the Skorokhod problem  $\mathbb{SP}(l,S)$. In fact, it is clear that $(X,A)\in S_G^p(0,T)\times I[0,T]$ and 
$$
\widehat{\mathbb{E}}[l(t,X_t)]=\widehat{\mathbb{E}}[l(t,S_t+A_t)]\geq \widehat{\mathbb{E}}[l(t,S_t+L_t(S_t))]\geq 0.
$$ 
By the definition of $A$, we have $A_t=L_t(S_t)$, $dA_t$-a.e. and $\mathbbm{1}_{\{L_t(S_t)=0\}}=0$ $dA_t$-a.e.. Noting that $\widehat{\mathbb{E}}[l(t,S_t+L_t(S_t))]=0$ when $L_t(S_t)>0$, we finally have
\begin{align*}
	\int_0^T \widehat{\mathbb{E}}[l(t,X_t)]dA_t=\int_0^T \widehat{\mathbb{E}}[l(t,S_t+L_t(S_t))]dA_t=\int_0^T\widehat{\mathbb{E}}[l(t,S_t+L_t(S_t))] \mathbbm{1}_{\{L_t(S_t)>0\}}dA_t=0.
\end{align*} 

	\smallskip

\noindent{\bf Step 2.} We now prove uniqueness. Suppose that $(X^1,A^1)$ and $(X^2,A^2)$ are two solutions to the Skorokhod problem  $\mathbb{SP}(l,S)$. Suppose that there exists some $t\in(0,T)$, such that $A^1_t<A^2_t$. Set 
\begin{align*}
\tau=\sup\Big\{u\leq t: A^1_u=A^2_u\Big\}.
\end{align*}
It is easy to check that for $u\in(\tau,t]$, $A^1_u<A^2_u$. 
Due to the strict increasing property of $l$ and Proposition \ref{comparison}, for any $u\in(\tau,t]$, we have
\begin{align*}
0\leq \widehat{\mathbb{E}}[l(u,S_u+A^1_u)]<
\widehat{\mathbb{E}}[l(u,S_u+A^2_u)].
\end{align*}
The flat-off condition (c) in Definition \ref{SPwithMR} implies that $dA^2=0$ on the interval $[\tau,t]$. It follows that 
\begin{align*}
A^2_{\tau}=A^2_t>A^1_t\geq A^1_{\tau},
\end{align*}
which contradicts the definition of $\tau$. The proof is complete.
\end{proof}\vspace{0.1in}

\subsection{Second construction of solutions}
\label{sec:second}
In this subsection, we offer an alternative construction of the solution to the Skorokhod problem with mean reflection. The advantage of this approach lies in its ability to establish a connection between the Skorokhod problem with mean reflection and the deterministic Skorokhod problem. However, a drawback is that it requires an additional assumption ($H'_l$) to achieve Theorem \ref{SP}, hence labeled as Theorem \ref{SP2}.
\begin{theorem}\label{SP2}
Suppose Assumptions ($H_l$) and ($H_S$) hold, as well as the following condition: 
\begin{itemize}
\item[($H'_l$)] There exist an increasing and continuous function $F:[0,\infty)\rightarrow [0,\infty)$ with $F(0)=0$ and two constants $0<c_l<C_l$, such that 
\begin{itemize}
\item[(1)] For any $t,s\in[0,T]$, $x\in\mathbb{R}$
\begin{align*}
|l(t,x)-l(s,x)|\leq F(|t-s|).
\end{align*}
\item[(2)] For any $t\in[0,T]$ and $x,y\in\mathbb{R}$, 
			\begin{equation}\label{bilip}
			c_l|x-y|\leq |l(t,x)-l(t,y)|\leq C_l|x-y|.
			\end{equation}
\end{itemize}
\end{itemize}
Then there exists a unique solution $(X,A)\in S_G^p(0,T)\times I[0,T]$ to the Skorokhod problem $\mathbb{SP}(l,S)$.
\end{theorem}

Both Assumptions ($H_l$) and ($H'_l$) impose regularity conditions on the function $l$. We delineate the comparison in the following remark.
 \begin{remark}\label{lem3.3}
  	Clearly, Assumption ($H'_l$) yields that $l$ is uniformly continuous in $(t,x)$, which is ($H_l$) (1). Next, suppose ($H'_l$) (2) holds true, and then for any $X\in L_G^1(\Omega_T)$, $L_t(X)$ is well-defined if we consider ($H_l$) but omit $\hE[\lim_{x\rightarrow \infty}l(t,x)]>0$ in ($H_l$) (3). In fact, for any $x\geq 0$,  ($H'_l$) (2) implies that 
	\begin{align*}
	l(t,x+X)-l(t,X)\geq c_l x.
	\end{align*}
	It follows that 
	\begin{align}\label{infinity}
	\lim_{x\rightarrow \infty}\hE[l(t,x+X)]\geq \lim_{x\rightarrow \infty}(\hE[l(t,X)]+c_l x)=\infty.
	\end{align}
		Furthermore, under Assumptions ($H_l$) and ($H'_l$) (2), Lemma 3.12 in \cite{LiuW} indicates that for any $X,Y\in L_G^1(\Omega_T)$, we have
  	\begin{align*}
  		|L_t(X)-L_t(Y)|\leq \frac{C_l}{c_l}\widehat{\mathbb{E}}|X-Y|,\quad \forall t\in[0,T].
  	\end{align*}
  \end{remark}

For any $t\in[0,T]$ and $Y\in L_G^1(\Omega_T)$, recalling that Proposition \ref{prop3.3} ensures that $l(t,Y-\hE[Y]+z)\in L_G^1(\Omega_T)$ for each fixed $z\in\mathbb{R}$, we define a mapping $H(t,\cdot,Y):\mathbb{R}\rightarrow \mathbb{R}$ by
\begin{align}\label{secondconstruction}
H(t,z,Y)=\hE[l(t,Y-\hE[Y]+z)].
\end{align}

\begin{lemma}\label{lemma2.2'}
Suppose that $l$ satisfies Assumptions ($H_l$) and ($H'_l$). 
Then, for any $t\in[0,T]$ and $Y\in L_G^1(\Omega_T)$, $H(t,\cdot,Y)$ is strictly increasing and continuous, with
\begin{displaymath}
\lim_{z\rightarrow -\infty}H(t,z,Y)=-\infty\quad\text{and}\quad \lim_{z\rightarrow +\infty}H(t,z,Y)=+\infty.
\end{displaymath}
\end{lemma}

\begin{proof}
For any $z,z'\in\mathbb{R}$, it is easy to check that 
\begin{align*}
&\hspace{-1cm}\Big|\hE[l(t,Y-\hE[Y]+z)]-\hE[l(t,Y-\hE[Y]+z')]\Big|\\
\leq &\hE\Big|l(t,Y-\hE[Y]+z)-l(t,Y-\hE[Y]+z')\Big|\leq C_l|z-z'|.
\end{align*}
Hence, $H(t,\cdot,Y)$ is continuous. Suppose that $z<z'$.  Proposition \ref{prop3.3} (ii) yields that 
\begin{align*}
H(t,z,Y)=\hE[l(t,Y-\hE[Y]+z)]< \hE[l(t,Y-\hE[Y]+z')]=H(t,z',Y),
\end{align*}
which implies that $H(t,\cdot,Y)$ is strictly increasing. The last assertion can be proved similarly with the help of equation \eqref{infinity}. The proof is complete. 
\end{proof}\vspace{0.1in}

By Lemma \ref{lemma2.2'}, we may define the inverse map $H^{-1}(t,\cdot,Y):\mathbb{R}\rightarrow \mathbb{R}$. In fact, for any $z\in\mathbb{R}$, 
\begin{align*}
H^{-1}(t,z,Y)=\bar{z} \Longleftrightarrow \hE[l(t,Y-\hE[Y]+\bar{z})]=z.
\end{align*}

\begin{lemma}\label{lemma2.2}
Suppose that $l$ satisfies Assumptions ($H_l$) and ($H'_l$), and $Y\in S_G^p(0,T)$ with $p\geq 1$. If $\bar{z}=\{\bar{z}_t\}_{t\in[0,T]}\in C[0,T]$, then $z=\{H(t,\bar{z}_t,Y_t)\}_{t\in[0,T]}\in C[0,T]$. Similarly, if $z=\{z_t\}_{t\in[0,T]}\in C[0,T]$, then $\bar{z}=\{H^{-1}(t,z_t,Y_t)\}_{t\in[0,T]}\in C[0,T]$.
\end{lemma}

\begin{proof}
For any $\bar{z}\in C[0,T]$ and any $s,t\in[0,T]$, it is easy to check that 
\begin{align*}
|z_t-z_s|\leq F(|t-s|)+C_l\Big\{|\bar{z}_t-\bar{z_s}|+2\hE|Y_t-Y_s|\Big\}.
\end{align*}
By Proposition \ref{the3.7}, we have $z\in C[0,T]$. It remains to prove the second statement. Given a sequence $\{t_n\}_{n\in\mathbb{N}}\subset [0,T]$ with $\lim_{n\rightarrow \infty}t_n=t$, we first claim that the sequence $\{\bar{z}_{t_n}\}$ is bounded. Otherwise, there exists a subsequence $\{t_{n_k}\}_{k\in\mathbb{N}}$ with $n_k\rightarrow\infty$ as $k\rightarrow\infty$, such that  $\lim_{k\rightarrow\infty}\bar{z}_{t_{n_k}}=\infty$. By Proposition \ref{the3.7} and the continuity property of $l$ in $(t,x)$, we have
\begin{align*}
\lim_{k\rightarrow \infty}\Big|\hE[l(t,Y_{t}-\hE[Y_{t}]+\bar{z}_{t_{n_k}})]-\hE[l(t_{n_k},Y_{t_{n_k}}-\hE[Y_{t_{n_k}}]+\bar{z}_{t_{n_k}})]\Big|=0.
\end{align*}
Lemma \ref{lemma2.2'} indicates that 
\begin{align*}
\lim_{k\rightarrow \infty}\hE[l(t,Y_{t}-\hE[Y_{t}]+\bar{z}_{t_{n_k}})]=\infty.
\end{align*}
Hence, we deduce that
\begin{align*}
z_t=\lim_{k\rightarrow \infty} z_{t_{n_k}}=\lim_{k\rightarrow \infty}\hE[l(t_{n_k},Y_{t_{n_k}}-\hE[Y_{t_{n_k}}]+\bar{z}_{t_{n_k}})]=\infty,
\end{align*}
which is a contradiction. To show that $\bar{z}_{t_n}\rightarrow \bar{z}_t$, it suffices to prove that for any subsequence $\{n'\}\subseteq\mathbb{N}$, one can choose a subsequence $\{n''\}\subseteq\{n'\}$ such that $\bar{z}_{t_{n''}}\rightarrow \bar{z}_t$. Since $\{\bar{z}_{t_{n'}}\}$ is bounded, there exists a subsequence $\{n''\}\subseteq{n'}$ such that $\bar{z}_{t_{n''}}\rightarrow z''$. By the definition of $\bar{z}$, we have
\begin{align*}
\hE[l(t_{n''},Y_{t_{n''}}-\hE[Y_{t_{n''}}]+\bar{z}_{t_{n''}})]=z_{t_{n''}}.
\end{align*}
Letting $n''$ go to infinity, we obtain that 
\begin{align*}
\hE[l(t,Y_t-\hE[Y_t]+\bar{z}'')]=z_t,
\end{align*}
which implies that $z''=\bar{z}_t$ by the definition of $\bar{z}$. The proof is complete. 
\end{proof}\vspace{0.1in}


In the following we give the proof of Theorem \ref{SP2}, which is  different to that of Theorem \ref{SP}. 
\smallskip

\begin{proof}[Proof of Theorem \ref{SP2}]
For any $t\in[0,T]$, set 
$$s_t=\hE[S_t]\quad \text{and}\quad\bar{l}_t=H^{-1}(t,0,S_t).$$ By Proposition \ref{the3.7} and Lemma \ref{lemma2.2}, we have $s=\{s_t\}_{t\in[0,T]}\in C[0,T]$ and $\bar{l}=\{\bar{l}_t\}_{t\in[0,T]}\in C[0,T]$. The trivial equality 
\begin{align*}
\hE[l(0,S_0-\hE[S_0]+\hE[S_0])]=\hE[l(0,S_0)]
\end{align*}
implies that $$\hE[S_0]=H^{-1}(0,\hE[l(0,S_0)],S_0).$$ Noting that $H^{-1}(t,\cdot,Y)$ is strictly increasing for any $t\in[0,T]$ and $Y\in L_G^1(\Omega_T)$, and $\hE[l(0,S_0)]\geq 0$, we have $\hE[S_0]\geq \bar{l}_0$. Now, let $(x,A)$ be the unique solution of the Skorokhod problem $\textrm{SP}(\bar{l},s)$, which is defined to satisfy the following conditions:
\begin{itemize}
\item[(1')] $x_t=s_t+A_t\geq \bar{l}_t$ for $t\in[0,T]$,
\item[(2')] $A\in I[0,T]$ and $\int_0^T(x_t-\bar{l}_t)dA_t=0$.
\end{itemize}
In fact, $A_t=\sup_{u\in[0,t]}(s_u-\bar{l}_u)^-$. Set $X_t=S_t+A_t$. We claim that $(X,A)$ is the solution to the Skorokhod problem  $\mathbb{SP}(l,S)$. First, simple calculation implies that
\begin{align*}
\hE[l(t,X_t)]=\hE[l(t,S_t+A_t)]=\hE[l(t,S_t-\hE[S_t]+x_t)]=H(t,x_t,S_t)\geq H(t,\bar{l}_t,S_t)=0.
\end{align*}
The above equation also indicates that $x_t=H^{-1}(t,\hE[l(t,X_t)],S_t)$. Furthermore, note that
\begin{align*}
x_t>\bar{l}_t\Longleftrightarrow \hE[l(t,S_t-\hE[S_t]+x_t)]>0 \Longleftrightarrow \hE[l(t,X_t)]>0,
\end{align*} 
which, together with the above condition (2'), implies that $\int_0^T \hE[l(t,X_t)]dA_t=0$.

It remains to prove the uniqueness. Suppose that $(X',A')$ is another solution of the Skorokhod problem  $\mathbb{SP}(l,S)$. Similar arguments as above could show that $(x',A')$ where $x'_t=H^{-1}(t,\hE[l(t,X'_t)],S_t)$, is the solution of the Skorokhod problem $\textrm{SP}(\bar{l},s)$. Due to the uniqueness of solutions to the classical Skorokhod problem, we have $A=A'$. Consequently, $X=X'$. The proof is complete.
\end{proof}\vspace{0.1in}

The preceding proof establishes the connection between the solution to a Skorokhod problem with mean reflection and the solution to a classical Skorokhod problem, thereby aiding in obtaining the following a priori estimates.

\begin{proposition}\label{proposition2.4}
	Suppose $l^i$ and $S^i$ satisfy Assumptions ($H_l$), ($H'_l$), and ($H_S$), and let $(X^i,A^i)$ be the solution of the Skorokhod problem  $\mathbb{SP}(l^i,S^i)$, for $i=1,2$. Then, there exists a constant $C$ depending on $c_l,C_l$, such that 
\begin{align*}
\sup_{t\in[0,T]}|A_t^1-A_t^2|&\leq C\Bigg\{\sup_{t\in[0,T]}\hE\sup_{x\in\mathbb{R}}|l^1(t,x)-l^2(t,x)|+\sup_{t\in[0,T]}\hE|S_t^1-S_t^2|\Bigg\},\\
\text{and}\quad\hE \sup_{t\in[0,T]}|X_t^1-X_t^2|&\leq C\Bigg\{\sup_{t\in[0,T]}\hE\sup_{x\in\mathbb{R}}|l^1(t,x)-l^2(t,x)|+\hE\sup_{t\in[0,T]}|S_t^1-S_t^2|\Bigg\}.
\end{align*}
\end{proposition}

\begin{proof}
It suffices to prove the first estimate, for the reason that the second estimate can be obtained by the representation of $X$ and the triangle inequality. By the proof of Theorem \ref{SP}, $(x^i,A^i)$ is the solution of the Skorokhod problem $\textrm{SP}(\bar{l}^i,s^i)$, where $s^i_t=\hE[S^i_t]$, $\bar{l}^i_t=H^{-1}_i(t,0,S^i_t)$ and $H^{-1}_i(t,\cdot,S^i_t)$ is the inverse map of $H_i(t,\cdot,S^i_t)=\hE[l^i(t,S^i_t-\hE[S^i_t]+\cdot)]$, for $i=1,2$. Since $$A^i_t=\sup_{u\in[0,t]}(s^i_u-\bar{l}^i_u)^-,$$ we have
\begin{align}\label{equation2.9}
\sup_{t\in[0,T]}|A_t^1-A_t^2|\leq \sup_{t\in[0,T]}|s_t^1-s^2_t|+\sup_{t\in[0,T]}|\bar{l}^1_t-\bar{l}^2_t|.
\end{align} 
It is easy to check that 
\begin{align*}
0=&\hE[l^1(t,S_t^1-\hE[S_t^1]+\bar{l}^1_t)]-\hE[l^2(t,S_t^2-\hE[S_t^2]+\bar{l}^2_t)]=: \mathcal{I}_t^1+\mathcal{I}_t^2+\mathcal{I}_t^3,
\end{align*}
where
\begin{align*}
\mathcal{I}_t^1=&\hE[l^1(t,S_t^1-\hE[S_t^1]+\bar{l}^1_t)]-\hE[l^2(t,S_t^1-\hE[S_t^1]+\bar{l}^1_t)],\\
\mathcal{I}_t^2=&\hE[l^2(t,S_t^1-\hE[S_t^1]+\bar{l}^1_t)]-\hE[l^2(t,S_t^2-\hE[S_t^2]+\bar{l}^1_t)],\\
\mathcal{I}_t^3=&\hE[l^2(t,S_t^2-\hE[S_t^2]+\bar{l}^1_t)]-\hE[l^2(t,S_t^2-\hE[S_t^2]+\bar{l}^2_t)].
\end{align*}
Without loss of generality, assume that $\bar{l}^1_t>\bar{l}^2_t$. By the assumption of $l^2$, for any $(t,x)\in[0,T]\times\mathbb{R}$, we have
\begin{align*}
c_l(\bar{l}^1_t-\bar{l}^2_t)\leq l^2(t,x+\bar{l}^1_t)-l^2(t,x+\bar{l}^2_t)\leq C_l(\bar{l}^1_t-\bar{l}^2_t).
\end{align*}
Consequently,
\begin{align*}
c_l(\bar{l}^1_t-\bar{l}^2_t)\leq \mathcal{I}_t^3\leq C_l(\bar{l}^1_t-\bar{l}^2_t).
\end{align*}
Hence, there exists a constant $C\in[c_l,C_l]$, such that $\mathcal{I}_t^3=C(\bar{l}^1_t-\bar{l}^2_t)$. The above analysis indicates that 
\begin{align*}
|\bar{l}^1_t-\bar{l}^2_t|\leq \frac{1}{c_l}|\mathcal{I}^3_t|\leq \frac{1}{c_l}(|\mathcal{I}_t^1|+|\mathcal{I}_t^2|)
\leq \frac{1}{c_l}\left\{\hE\sup_{x\in\mathbb{R}}|l^1(t,x)-l^2(t,x)|+2C_l\hE|S_t^1-S_t^2|\right\}.
\end{align*}
Plugging the above inequality into equation \eqref{equation2.9}, we obtain the desired result.
\end{proof}\vspace{0.1in}

\begin{corollary}\label{corollary2.5}
Suppose $l$ satisfies Assumptions ($H_l$) and ($H'_l$), and $S$ satisfies Assumption ($H_S$). Let $(X,A)$ be the solution of the Skorokhod problem  $\mathbb{SP}(l,S)$. For any $0\leq s\leq t\leq T$, we have
\begin{align*}
|A_t-A_s|&\leq C\left\{\sup_{r\in[s,t]}\hE|S_r-S_s|+F(|t-s|)\right\}, \\
\text{and}\quad\hE \sup_{r\in[s,t]}|X_r-X_s|&\leq C\left\{\hE\sup_{r\in[s,t]}|S_r-S_s|+F(|t-s|)\right\}.
\end{align*}
\end{corollary}

\begin{proof}
We only prove the first estimate, while the second estimate can be obtained by the representation of $X$ and the triangle inequality. For any fixed $s\in[0,T]$ and $r\in[0,T]$, set 
$$S'_r=S_{r\wedge s} \quad \text{and}\quad l'(r,x)=l(r\wedge s,x).$$ Let $(X',A')$ be the solution to the Skorokhod problem  $\mathbb{SP}(l',S')$. It is easy to check that for any $r\in[0,T]$, $$X'_r=X_{r\wedge s}\quad \text{and}\quad A'_r=A_{r\wedge s}.$$ By Proposition \ref{proposition2.4}, we have
\begin{align*}
|A_t-A_s|&\leq \sup_{r\in[s,t]}|A_r-A'_r|\\
&\leq C\left\{\sup_{r\in[s,t]}\hE\sup_{x\in\mathbb{R}}|l(r,x)-l'(r,x)|+\sup_{r\in[s,t]}\hE|S_r-S'_r|\right\}\\
&\leq C\left\{\sup_{r\in[s,t]}\hE|S_r-S_s|+\sup_{r\in[s,t]}F(|r-s|)\right\}\\
&\leq C\left\{\sup_{r\in[s,t]}\hE|S_r-S_s|+F(|t-s|)\right\},
\end{align*}
as desired.
\end{proof}\vspace{0.1in}

\section{Mean reflected $G$-SDEs}
\label{sec:MRGSDEs}
In this section, we establish the well-posedness of the mean reflected $G$-SDE \eqref{MRGSDE}, recalled here as follows: for $t\in[0,T]$,
\begin{equation*}
		\begin{cases}
	X_t=x_0+\int_0^t b(s,X_s)ds+\int_0^t h(s,X_s)d\langle B\rangle_s+\int_0^t \sigma(s,X_s)dB_s+A_t,\vspace{0.2cm}\\
	\hE[l(t,X_t)]\geq 0\quad\text{and}\quad
	\int_0^T \hE[l(t,X_t)]dA_t=0.
\end{cases}
\end{equation*}
We consider the coefficient functions $b,h,\sigma:\Omega_T\times [0,T]\times\mathbb{R}\rightarrow\mathbb{R}$ satisfy the following conditions:
\begin{itemize}
\item[(A1)] For some $p\geq 1$ and each $x\in\mathbb{R}$, $b(\cdot,\cdot,x),h(\cdot,\cdot,x)\in M_G^p(0,T)$ and $\sigma(\cdot,\cdot,x)\in H_G^p(0,T)$.
\item[(A2)] For any $(\omega,t)\in\Omega_T\times[0,T]$ and any $x,x'\in \mathbb{R}$, there exists a constant $\kappa>0$, such that 
$$
|b(\omega,t,x)-b(\omega,t,x')|+|h(\omega,t,x)-h(\omega,t,x')|+|\sigma(\omega,t,x)-\sigma(\omega,t,x')|\leq \kappa|x-x'|.
$$
\end{itemize}

Now, we state the main result of this section.
\begin{theorem}\label{main}
Suppose Assumptions (A1), (A2), ($H_l$) and ($H'_l$) hold, and $\widehat{\mathbb{E}}[l(0,x_0)]\geq 0$. The mean reflected $G$-SDE \eqref{MRGSDE} admits a unique pair of solution $(X,A)\in S_G^p(0,T)\times I[0,T]$. 
\end{theorem}

\begin{proof}
First, we fix a positive constant $\delta$, which will be determined later. Given a process $U\in S_G^p(0,\delta)$, for any $t\in[0,\delta]$, set 
\begin{displaymath}
\widetilde{X}_t=x_0+\int_0^t b(s,U_s)ds+\int_0^t h(s,U_s)d\langle B\rangle_s+\int_0^t \sigma(s,U_s)dB_s.
\end{displaymath} 
By Assumptions (A1) and (A2), H\"{o}lder's inequality and the BDG inequality  under $G$-expectation (Proposition \ref{BDG}), we may check that $\widetilde{X}\in S_G^p(0,\delta)$. Theorem \ref{SP} indicates that there exists a unique solution $(X,A)$ to the Skorokhod problem $\mathbb{SP}(l,\widetilde{X})$ on  the time interval $[0,\delta]$. We define the mapping $\Gamma: S_G^p(0,\delta)\rightarrow S_G^p(0,T)$ as
\begin{displaymath}
\Gamma(U)=X.
\end{displaymath}
We then show that $\Gamma$ is a contraction mapping when $\delta$ is small enough. Similarly, for given $U'\in S_G^p(0,\delta)$, define $\widetilde{X}'$ as above. Let $(X,A)$ and $(X',A')$ be the solutions to the Skorokhod problem associated with $(l,\widetilde{X})$ and $(l,\widetilde{X}')$, respectively. We define
\begin{align*}
	\widehat{b}_t=b(t,U_t)-b(t,U'_t), &\quad 	\widehat{h}_t=h(t,U_t)-h(t,U'_t), \quad 	\widehat{\sigma}_t=\sigma(t,U_t)-\sigma(t,U'_t), \\
&\widehat{A}_t=A_t-A'_t\quad \text{and}\quad\widehat{X}_t=X_t-X'_t.
\end{align*}
Simple calculation gives that 
\begin{equation}\label{e1}\begin{split}
\widehat{\mathbb{E}}\sup_{t\in[0,\delta]}|\widehat{X}_t|^p\leq &C\Bigg\{\widehat{\mathbb{E}}\sup_{t\in[0,\delta]}|\widetilde{X}_t-\widetilde{X}'_t|^p+\sup_{t\in[0,\delta]}|\widehat{A}_t|^p\Bigg\}\\
\leq &C\Bigg\{\widehat{\mathbb{E}}\sup_{t\in[0,\delta]}\Bigg|\int_0^t \widehat{b}_s ds\Bigg|^p+\widehat{\mathbb{E}}\sup_{t\in[0,\delta]}\Bigg|\int_0^t \widehat{h}_s d\langle B\rangle_s\Bigg|^p\\
&\hspace{4cm}+\widehat{\mathbb{E}}\sup_{t\in[0,\delta]}\Bigg|\int_0^t \widehat{\sigma}_s dB_s\Bigg|^p+\sup_{t\in[0,\delta]}|\widehat{A}_t|^p\Bigg\}\\
\leq &C\Bigg\{\widehat{\mathbb{E}}\int_0^{\delta} |\widehat{U}_t|^p ds+\sup_{t\in[0,\delta]}|\widehat{A}_t|^p \Bigg\},
\end{split}\end{equation}
where we used H\"{o}lder's inequality and the BDG inequality  under $G$-expectation (Proposition \ref{BDG}). Here, $C$ is a constant depending only on $p,\delta,\kappa,\overline{\sigma},\underline{\sigma}$. Recalling the first proof of Theorem \ref{SP}, we have 
\begin{equation}\label{e2}
\sup_{t\in[0,\delta]}|\widehat{A}_t|^p\leq \sup_{t\in[0,\delta]}|L_t(\widetilde{X}_t)-L_t(\widetilde{X}'_t)|^p
\leq \frac{C_l^p}{c_l^p}\sup_{t\in[0,\delta]}\widehat{\mathbb{E}}|\widetilde{X}_t-\widetilde{X}'_t|^p\leq C\widehat{\mathbb{E}}\int_0^{\delta} |\widehat{U}_t|^p ds,
\end{equation}
where we used Remark \ref{lem3.3} in the second inequality and $C$ is a constant depending on $p,\delta,\kappa$, $\overline{\sigma},\underline{\sigma}$, $C_l,c_l$. Equations \eqref{e1} and \eqref{e2} yield that 
\begin{displaymath}
\widehat{\mathbb{E}}\sup_{t\in[0,\delta]}|\widehat{X}_t|^p\leq C\delta \widehat{\mathbb{E}}\sup_{t\in[0,\delta]}|\widehat{U}_t|^p.
\end{displaymath}  
Therefore, for $\delta$ being sufficiently small, $\Gamma$ is a contraction mapping. We then obtain the existence and uniqueness of the solution, denoted by $(X^{(1)},A^{(1)})$, to the mean reflected $G$-SDE \eqref{MRGSDE} on the interval $[0,\delta]$. Now, let $N$ be such that $N=[\frac{T}{\delta}]+1$. For any $2\leq n\leq N$, by a similar analysis as above, the following reflected $G$-SDEs on the time interval $[(n-1)\delta,n\delta\wedge T]$ admits a unique solution $(X^{(n)},A^{(n)})$:
\begin{displaymath}
	\begin{cases}
		X^{(n)}_t=X^{(n-1)}_{(n-1)\delta}+\int_{(n-1)\delta}^t b(s,X^{(n)}_s)ds+\int_{(n-1)\delta}^t h(s,X^{(n)}_s)d\langle B\rangle_s+\int_{(n-1)\delta}^t \sigma(s,X^{(n)}_s)dB_s+A^{(n)}_t,\vspace{0.3cm}\\
		\widehat{\mathbb{E}}[l(t,X^{(n)}_t)]\geq0 \quad\text{and} \quad\int_{(n-1)\delta}^{n\delta\wedge T} \widehat{\mathbb{E}}[l(t,X^{(n)}_t)]dA^n_t=0.
	\end{cases}
\end{displaymath}
We define, for $t\in((n-1)\delta,n\delta\wedge T]$ and $1\leq n\leq N$, 
\begin{align*}
X_t=X^{(1)}_t \mathbbm{1}_{[0,\delta]}(t)+\sum_{n=2}^N X^{(n)}_t \mathbbm{1}_{((n-1)\delta,n\delta \wedge T]}(t)\quad\text{and}\quad
 A_t=A^{(n)}_t+\sum_{j=1}^{n-1} A^{(j)}_{jh}, 
\end{align*}
with the convention that $\sum_{j=1}^0 A^j_{jh}=0$. It is easy to check that $(X,A)$ is the solution to the mean reflected $G$-SDE \eqref{MRGSDE}. The uniqueness follows from the uniqueness for each small interval. The proof is complete.
\end{proof}\vspace{0.1in}

Let $(X,A)$ be the solution to the mean reflected $G$-SDE  \eqref{MRGSDE} and we provide its moment estimates in the proposition below. Set 
\begin{align}\label{U}
U_t=x_0+\int_0^t b(s,X_s)ds+\int_0^t h(s,X_s)d\langle B\rangle_s+\int_0^t \sigma(s,X_s)dB_s.
\end{align}
\begin{proposition}\label{prop2.6}
Suppose Assumptions (A1), (A2), ($H_l$), and ($H'_l$) hold, and $\hE[l(0,x_0)]\geq 0$. Then, there exists a constant $C$ depending on $p,\kappa,C_l,c_l,\overline{\sigma},\underline{\sigma},T$, such that 
\begin{align}\label{estimateX}
&\widehat{\mathbb{E}}\sup_{t\in[0,T]}|X_t|^p\\
&\leq C\Bigg(1+|x_0|^p+\hE\Bigg[\int_0^T|b(s,0)|^pds\Bigg]+\hE\Bigg[\int_0^T |h(s,0)|^p ds\Bigg]
+\hE\Bigg(\int_0^T |\sigma(s,0)|^2ds\Bigg)^{p/2}\Bigg).\nonumber
\end{align}
Furthermore, suppose for each $x\in\mathbb{R}$, $b(\cdot,\cdot,x),h(\cdot,\cdot,x),\sigma(\cdot,\cdot,x)\in S_G^p(0,T)$.
Then, for any $0\leq s\leq t\leq T$,
\begin{align}\label{estimateA}
|A_t-A_s|\leq C\Big\{|t-s|^{\frac{1}{2}}+F(|t-s|)\Big\}\quad\text{and}\quad  \hE|X_t-X_s|^p\leq C\Big\{|t-s|^{\frac{p}{2}}+F^p(|t-s|)\Big\}.
\end{align}
\end{proposition}

\begin{proof}
Note that $A$ can be viewed as the second component of the solution to the Skorokhod problem  $\mathbb{SP}(l,U)$. By Corollary \ref{corollary2.5}, for any $0\leq s\leq t\leq T$, we have 
\begin{align}
&|A_t|\leq C \Bigg\{\sup_{r\in[0,t]}\hE|U_r-U_0|+F(|t|)\Bigg\}\leq C\Bigg\{\sup_{r\in[0,t]}\hE|U_r-x_0|+F(|T|)\Bigg\},\\
&|A_t-A_s|\leq C\Bigg\{\sup_{r\in[s,t]}\hE|U_r-U_s|+F(|t-s|)\Bigg\}. \label{atas}
\end{align}
Simple calculation gives that 
\begin{align*}
\hE\sup_{s\in[0,t]}|X_s|^p&\leq C\Bigg\{\hE\sup_{s\in[0,t]}|U_s|^p+|A_t|^p\Bigg\}\\
&\leq C\Bigg\{1+|x_0|^p+\hE\Bigg[\int_0^t|b(s,0)|^pds\Bigg]+\hE\Bigg[\int_0^t |h(s,0)|^p ds\Bigg]\\
&\hspace{3cm}+\hE\Bigg(\int_0^t |\sigma(s,0)|^2ds\Bigg)^{p/2}+\hE\Bigg[\int_0^t |X_s|^p ds\Bigg]\Bigg\}\\
&\leq C\Bigg\{1+|x_0|^p+\hE\Bigg[\int_0^t|b(s,0)|^pds\Bigg]+\hE\Bigg[\int_0^t |h(s,0)|^p ds\Bigg]\\
&\hspace{3cm}+\hE\Bigg(\int_0^t |\sigma(s,0)|^2ds\Bigg)^{p/2}+\int_0^t \hE\Bigg[\sup_{r\in[0,s]}|X_r|^p\Bigg] ds\Bigg\}.
\end{align*}
Applying Gr\"onwall's inequality, we obtain the moment estimate \eqref{estimateX}.

Next, by H\"{o}lder's inequality and Proposition \ref{BDG}, we have
\begin{align*}
\hE|U_t-U_s|^p\leq &C\Bigg\{\hE\left|\int_s^t b(r,X_r)dr\right|^p+\hE\left|\int_s^t h(r,X_r)d\langle B\rangle_r \right|^p+\hE\left|\int_s^t \sigma(r,X_r)dB_r\right|^p\Bigg\}\\
\leq &C\Bigg\{\hE\left|\int_s^t |b(r,X_r)| dr\right|^p+\hE\left|\int_s^t |h(r,X_r)| dr\right|^p+\hE\left(\int_s^t |\sigma(r,X_r)|^2 dr\right)^{p/2}\Bigg\}\\
\leq &C\Bigg\{\hE\left|\int_s^t \Big(\sup_{r\in[s,t]}|b(r,0)|+\sup_{r\in[s,t]}|h(r,0)|+\sup_{r\in[s,t]}|X_r|\Big)dr\right|^p\\
 &\hspace{3cm}+\hE\left|\int_s^t \Big(\sup_{r\in[s,t]}|\sigma(r,0)|^2+\sup_{r\in[s,t]}|X_r|^2\Big)dr\right|^{p/2}\Bigg\}\\
 \leq &|t-s|^{p/2},
\end{align*}
where we used the estimate \eqref{estimateX} and the fact that $b,h,\sigma\in S_G^p(0,T)$ in the last step. Plugging the above inequality into equation \eqref{atas}, we obtain the first result in equation \eqref{estimateA}. Noting that $$X_t-X_s=U_t-U_s+A_t-A_s,$$ we obtain the estimate for $\hE|X_t-X_s|^p$. The proof is complete.
\end{proof}\vspace{0.1in}


Proposition \ref{prop2.6} establishes the continuity of the second component $A$ in the solution to the mean reflected $G$-SDE.  In the following, we show that subject to certain regularity conditions on the loss function $l$, the function $A$ exhibits Lipschitz continuity. Let $C^{1,2}_b([0,T]\times\mathbb{R})$ be the set of functions on $[0,T]$ possessing continuous first and second partial derivatives, and bounded partial derivatives up to second order.
\begin{proposition}\label{prop2.7}
 Let Assumptions (A1) and (A2) hold with $p\geq 2$. Assume that $l\in C^{1,2}_b([0,T]\times \mathbb{R})$ is bi-Lipschtiz (i.e., satisfies equation \eqref{bilip}) and strictly increasing in its second component with $l(0,x_0)\geq 0$ and 
\begin{align*}
\sup_{t\in[0,T]}\Big(\hE|b(t,0)|^2+\hE|h(t,0)|^2+\hE|\sigma(t,0)|^2\Big)<\infty.
\end{align*}
 Let $(X,A)$ be the solution to the mean reflected $G$-SDE \eqref{MRGSDE}. Then the function $A$ is Lipschitz continuous.
\end{proposition}

\begin{proof}
	Define the operator $\widetilde{L}_t:L_G^1(\Omega_T)\rightarrow \mathbb{R}$ as 
\begin{align}\label{construction'}
	\widetilde{L}_t(X)=\inf\Big\{x\in \mathbb{R}:\widehat{\mathbb{E}}[l(t,x+X)]\geq 0\Big\}.
\end{align}
Clearly, we have $L_t(X)=(\widetilde{L}_t(X))^+$. Let $(X,A)$ be the solution to the Skorokhod problem $\mathbb{SP}(l,S)$. By the proof of Theorem \ref{SP}, we have $A_t=\sup_{s\in[0,T]}(\widetilde{L}_s(S_s))^+$. We first prove that $t\rightarrow \widetilde{L}_t(U_t)$ is Lipschitz continuous on $[0,T]$. In fact, by the definition of $\widetilde{L}_t(U_t)$, we have 
\begin{align*}
\hE[l(t,\widetilde{L}_t(U_t)+U_t)]=0.
\end{align*}
If $x\geq y$, since $l$ is bi-Lipschitz and $\hE[\cdot]$ is sub-additive, we obtain that 
\begin{align*}
c_l(x-y)\leq -\hE[l(t,y+U_t)-l(t,x+U_t)]\leq \hE[l(t,x+U_t)]-\hE[l(t,y+U_t)].
\end{align*}
The above analysis implies that, for any $0\leq s<t\leq T$,
\begin{align*}
|\widetilde{L}_s(U_s)-\widetilde{L}_t(U_t)|&\leq \frac{1}{c_l}\Big|\hE[l(t,\widetilde{L}_s(U_s)+U_t)]-\hE[l(t,\widetilde{L}_t(U_t)+U_t)]\Big|\nonumber\\
&= \frac{1}{c_l}\Big|\hE[l(t,\widetilde{L}_s(U_s)+U_t)]-\hE[l(s,\widetilde{L}_s(U_s)+U_s)]\Big|.
\end{align*}
Applying $G$-It\^{o}'s formula \citep{lp} to $l(t,\widetilde{L}_s(U_s)+U_t)$, we obtain that 
\begin{align*}
l(t,\widetilde{L}_s(U_s)+U_t)=&l(s,\widetilde{L}_s(U_s)+U_s)+\int_s^t \Big[\partial_t l(r,\widetilde{L}_s(U_s)+U_r)+\partial_x l(r,\widetilde{L}_s(U_s)+U_r)b(r,X_r)\Big]dr\\
&+\int_s^t \Big[\partial_x l(r,\widetilde{L}_s(U_s)+U_r) h(r,X_r)+\frac{1}{2}\partial_x^2 l(r,\widetilde{L}_s(U_s)+U_r) \sigma^2(r,X_r)\Big]d\langle B\rangle_r\\
&+\int_s^t \partial_x l(r,\widetilde{L}_s(U_s)+U_r)\sigma(r,X_r)dB_r.
\end{align*}
By the above two equations  and the estimate \eqref{estimateX} in Proposition \ref{prop2.6}, we have
\begin{align*}
&|\widetilde{L}_s(U_s)-\widetilde{L}_t(U_t)|\\
&\leq \frac{1}{c_l}\hE\Bigg|\int_s^t \Big[\partial_t l(r,\widetilde{L}_s(U_s)+U_r)+\partial_x l(r,\widetilde{L}_s(U_s)+U_r)b(r,X_r)\Big]dr\\
&\hspace{1.5cm}+\int_s^t \Big[\partial_x l(r,\widetilde{L}_s(U_s)+U_r) h(r,X_r)+\frac{1}{2}\partial_x^2 l(r,\widetilde{L}_s(U_s)+U_r) \sigma^2(r,X_r)\Big]d\langle B\rangle_r \Bigg|\\
&\leq \frac{C}{c_l}\hE\Bigg[\int_s^t \Big(1+|b(r,0)|+|h(r,0)|+|\sigma(r,0)|^2+|X_r|+|X_r|^2\Big)dr \Bigg]\\
&\leq C|t-s|.
\end{align*}

Now we are ready to prove that $A$ is Lipschitz continuous. Since $L_t(U_t)=(\widetilde{L}_t(U_t))^+$, $L_\cdot(U_\cdot)$ is Lipschitz continuous with Lipschitz constant $C$. Note that $A$ can be viewed as the second component of the solution to the Skorokhod problem  $\mathbb{SP}(l,U)$. By the proof of Theorem \ref{SP}, for any $0\leq s\leq t\leq T$, we have
\begin{align*}
A_t=\sup_{r\in[0,t]}L_r(U_r)=\max\Bigg(\sup_{r\in[0,s]}L_r(U_r),\;\sup_{r\in[s,t]}L_r(U_r)\Bigg)=\max\Bigg(A_s,\;\sup_{r\in[s,t]}L_r(U_r)\Bigg).
\end{align*}
For any $r\in[s,t]$, it is easy to check that 
\begin{align*}
L_r(U_r)\leq L_s(U_s)+C(r-s)\leq A_s+C(t-s),
\end{align*}
which implies that $\sup_{r\in[s,t]}L_r(U_r)\leq A_s+C(t-s)$. Therefore,
\begin{align*}
0\leq A_s\leq A_t\leq A_s+C(t-s).
\end{align*}
The proof is complete.
\end{proof}\vspace{0.1in}

\section*{Acknowledgments}
	The research  was supported  by the National Natural Science Foundation of China (No. 12301178), the Natural Science Foundation of Shandong Province for Excellent Young Scientists Fund Program (Overseas) (No. 2023HWYQ-049),  the Fundamental Research Funds for the Central Universities, the Natural Science Foundation of Shandong Province (No. ZR2023ZD35) and  the Qilu Young Scholars Program of Shandong University.





\bibliography{bib-ms}

\begin{thebibliography}{}

\bibitem[Artzner et~al., 1999]{ADEH}
Artzner, P., Delbaen, F., Eber, J.-M., and Heath, D. (1999).
\newblock Coherent measures of risk.
\newblock {\em Math. Finance}, 9:203--228.

\bibitem[Bally et~al., 2002]{bally2002reflected}
Bally, V., Caballero, M., Fernandez, B., and El~Karoui, N. (2002).
\newblock {\em Reflected {BSDE}'s, {PDE}'s and Variational Inequalities}.
\newblock PhD thesis, INRIA.

\bibitem[Bouchard et~al., 2015]{Bouchard2015BSDEs}
Bouchard, B., Elie, R., and R{\'e}veillac, A. (2015).
\newblock {BSDE}s with weak terminal condition.
\newblock {\em The Annals of Probability}, 43(2):572 -- 604.

\bibitem[Briand et~al., 2020a]{Briand2020Particles}
Briand, P., de~Raynal, P.-E.~C., Guillin, A., and Labart, C. (2020a).
\newblock Particles systems and numerical schemes for mean reflected stochastic
  differential equations.
\newblock {\em The Annals of Applied Probability}, 30(4):1884--1909.

\bibitem[Briand et~al., 2018]{briand2018bsdes}
Briand, P., Elie, R., and Hu, Y. (2018).
\newblock {BSDE}s with mean reflection.
\newblock {\em The Annals of Applied Probability}, 28(1):482--510.

\bibitem[Briand et~al., 2020b]{briand2020mean}
Briand, P., Ghannoum, A., and Labart, C. (2020b).
\newblock Mean reflected stochastic differential equations with jumps.
\newblock {\em Advances in Applied Probability}, 52(2):523--562.

\bibitem[Briand and Hibon, 2021]{briand2021particles}
Briand, P. and Hibon, H. (2021).
\newblock Particles systems for mean reflected {BSDE}s.
\newblock {\em Stochastic Processes and their Applications}, 131:253--275.

\bibitem[Burdzy et~al., 2009]{BKR}
Burdzy, K., Kang, W., and Ramanan, K. (2009).
\newblock The {S}korokhod problem in a time-dependent interval.
\newblock {\em Stochastic Processes and their Applications}, 119:428--452.

\bibitem[Chaleyat-Maurel and El~Karoui, 1978]{CE}
Chaleyat-Maurel, M. and El~Karoui, N. (1978).
\newblock Un problème de réflexion et ses applications au temps local et aux
  équations différentielles stochastiques sur $\mathbb{R}$, cas continu.
\newblock {\em Astérisque}, 52-53:117--144.

\bibitem[Cheng and Riedel, 2013]{cheng2013optimal}
Cheng, X. and Riedel, F. (2013).
\newblock Optimal stopping under ambiguity in continuous time.
\newblock {\em Mathematics and Financial Economics}, 7:29--68.

\bibitem[Denis et~al., 2011]{DHP11}
Denis, L., Hu, M., and Peng, S. (2011).
\newblock Function spaces and capacity related to a sublinear expectation:
  application to ${G}$-{B}rownian motion paths.
\newblock {\em Potential Anal.}, 34:139--161.

\bibitem[El~Karoui et~al., 1997a]{el1997reflected}
El~Karoui, N., Kapoudjian, C., Pardoux, E., Peng, S., and Quenez, M.-C.
  (1997a).
\newblock Reflected solutions of backward {SDE}'s, and related obstacle
  problems for {PDE}'s.
\newblock {\em the Annals of Probability}, 25(2):702--737.

\bibitem[El~Karoui et~al., 1997b]{el1997reflected2}
El~Karoui, N., Pardoux, {\'E}., and Quenez, M.~C. (1997b).
\newblock Reflected backward {SDE}s and {A}merican options.
\newblock {\em Numerical methods in finance}, 13:215--231.

\bibitem[Falkowski and S{\l}omi{\'n}ski, 2021]{falkowski2021mean}
Falkowski, A. and S{\l}omi{\'n}ski, L. (2021).
\newblock Mean reflected stochastic differential equations with two
  constraints.
\newblock {\em Stochastic Processes and their Applications}, 141:172--196.

\bibitem[Falkowski and S{\l}omi{\'n}ski, 2022]{falkowski2022backward}
Falkowski, A. and S{\l}omi{\'n}ski, L. (2022).
\newblock Backward stochastic differential equations with mean reflection and
  two constraints.
\newblock {\em Bulletin des Sciences Math{\'e}matiques}, 176:103117.

\bibitem[Gao, 2009]{G}
Gao, F. (2009).
\newblock Pathwise properties and homeomorphic flows for stochastic
  differential equations driven by ${G}$-{B}rownian motion.
\newblock {\em Stochastic Process. Appl.}, 119:3356--3382.

\bibitem[Hamadene and Zhang, 2010]{hamadene2010switching}
Hamadene, S. and Zhang, J. (2010).
\newblock Switching problem and related system of reflected backward {SDE}s.
\newblock {\em Stochastic Processes and their applications}, 120(4):403--426.

\bibitem[Hibon et~al., 2018]{hibon2017quadratic}
Hibon, H., Hu, Y., Lin, Y., Luo, P., and Wang, F. (2018).
\newblock Quadratic {BSDE}s with mean reflection.
\newblock {\em Mathematical Control and Related Fields}, 8(3-4):721--738.

\bibitem[Hu et~al., 2014]{HJPS2}
Hu, M., Ji, S., Peng, S., and Song, Y. (2014).
\newblock {Comparison theorem, Feynman-Kac formula and Girsanov transformation
  for {BSDE}s driven by ${G}$-{B}rownian motion}.
\newblock {\em Stochastic Processes and their Applications}, 124:1170--1195.

\bibitem[Li, 2023]{li2023backward}
Li, H. (2023).
\newblock Backward stochastic differential equations with double mean
  reflections.
\newblock {\em arXiv preprint arXiv:2307.05947}.

\bibitem[Li and Ning, 2024a]{li2024doubly}
Li, H. and Ning, N. (2024a).
\newblock Doubly reflected backward {SDE}s driven by {$G$}-{B}rownian motions
  and fully nonlinear {PDE}s with double obstacles.
\newblock {\em arXiv:2008.09973}.

\bibitem[Li and Ning, 2024b]{li2024propagation}
Li, H. and Ning, N. (2024b).
\newblock Propagation of chaos for doubly mean reflected {BSDE}s.
\newblock {\em arXiv preprint arXiv:2401.16617}.

\bibitem[Li and Peng, 2020]{li2017reflected}
Li, H. and Peng, S. (2020).
\newblock Reflected backward stochastic differential equation driven by
  {$G$-B}rownian motion with an upper obstacle.
\newblock {\em Stochastic Processes and their Applications},
  130(11):6556--6579.

\bibitem[Li et~al., 2018a]{LPS}
Li, H., Peng, S., and Song, Y. (2018a).
\newblock Supermartingale decomposition theorem under ${G}$-expectation.
\newblock {\em Electron. J. Probab.}, 23:1--20.

\bibitem[Li et~al., 2018b]{li2018reflected}
Li, H., Peng, S., and Soumana~Hima, A. (2018b).
\newblock Reflected solutions of backward stochastic differential equations
  driven by {$G$-B}rownian motion.
\newblock {\em Science China Mathematics}, 61:1--26.

\bibitem[Li and Song, 2021]{li2021backward}
Li, H. and Song, Y. (2021).
\newblock Backward stochastic differential equations driven by {$G$-B}rownian
  motion with double reflections.
\newblock {\em Journal of Theoretical Probability}, 34:2285--2314.

\bibitem[Li and Lin, 2017]{LL}
Li, X. and Lin, Y. (2017).
\newblock Strict comparison theorems under sublinear expectations.
\newblock {\em Arch. Math.}, 109:489--498.

\bibitem[Li and Peng, 2011]{lp}
Li, X. and Peng, S. (2011).
\newblock Stopping times and related {I}t\^{o}'s calculus with ${G}$-{B}rownian
  motion.
\newblock {\em Stochastic Process. Appl.}, 121:1492--1508.

\bibitem[Li, 2018]{li2018large}
Li, Y. (2018).
\newblock Large deviation principle for the mean reflected stochastic
  differential equation with jumps.
\newblock {\em Journal of Inequalities and Applications}, 2018:1--15.

\bibitem[Lin, 2013]{Lin}
Lin, Y. (2013).
\newblock Stochastic differential equations driven by ${G}$-{B}rownian motion
  with reflecting boundary conditions.
\newblock {\em Electron. J. Probab.}, 18:1--24.

\bibitem[Lin and Soumana~Hima, 2019]{LSH}
Lin, Y. and Soumana~Hima, A. (2019).
\newblock Reflected stochastic differential equations driven by
  ${G}$-{B}rownian motion in non-convex domains.
\newblock {\em Stochastics and Dynamics}, 19(3):1950025.

\bibitem[Lions and Sznitman, 1984]{LS}
Lions, P.~L. and Sznitman, A.~L. (1984).
\newblock Stochastic differential equations with reflecting boundary
  conditions.
\newblock {\em Comm. Pure Appl. Math.}, 37:511--537.

\bibitem[Liu and Wang, 2019]{LiuW}
Liu, G. and Wang, F. (2019).
\newblock {BSDE}s with mean reflection driven by ${G}$-{B}rownian motion.
\newblock {\em Journal of Mathematical Analysis and Applications},
  470:599--618.

\bibitem[Ma and Zhang, 2005]{ma2005representations}
Ma, J. and Zhang, J. (2005).
\newblock Representations and regularities for solutions to {BSDE}s with
  reflections.
\newblock {\em Stochastic processes and their applications}, 115(4):539--569.

\bibitem[Ning and Wu, 2021]{ning2021well}
Ning, N. and Wu, J. (2021).
\newblock Well-posedness and stability analysis of two classes of generalized
  stochastic volatility models.
\newblock {\em SIAM Journal on Financial Mathematics}, 12(1):79--109.

\bibitem[Ning and Wu, 2023]{ning2023multi}
Ning, N. and Wu, J. (2023).
\newblock Multi-dimensional path-dependent forward-backward stochastic
  variational inequalities.
\newblock {\em Set-Valued and Variational Analysis}, 31(1):2.

\bibitem[Ning et~al., 2024]{NING2024One}
Ning, N., Wu, J., and Zheng, J. (2024).
\newblock One-dimensional {McKean–Vlasov} stochastic variational inequalities
  and coupled {BSDE}s with locally hölder noise coefficients.
\newblock {\em Stochastic Processes and their Applications}, 171:104315.

\bibitem[Peng, 2007]{P07a}
Peng, S. (2007).
\newblock ${G}$-expectation, ${G}$-{B}rownian motion and related stochastic
  calculus of {I}t\^o type.
\newblock In {\em Stochastic analysis and applications, Abel Symp.}, volume~2,
  pages 541--567. Springer, Berlin.

\bibitem[Peng, 2008]{P08a}
Peng, S. (2008).
\newblock Multi-dimensional ${G}$-{B}rownian motion and related stochastic
  calculus under ${G}$-expectation.
\newblock {\em Stochastic Processes and their Applications},
  118(12):2223--2253.

\bibitem[Peng, 2019]{P19}
Peng, S. (2019).
\newblock {\em Nonlinear Expectations and Stochastic Calculus Under
  Uncertainty: With Robust {CLT} and ${G}$-{B}rownian Motion}, volume~95 of
  {\em Probability Theory and Stochastic Modelling}.
\newblock Springer.

\bibitem[Qu and Wang, 2023]{qu2023multi}
Qu, B. and Wang, F. (2023).
\newblock Multi-dimensional {BSDE}s with mean reflection.
\newblock {\em Electronic Journal of Probability}, 28:1--26.

\bibitem[{S}korokhod, 1961]{Skorokhod1}
{S}korokhod, A. (1961).
\newblock Stochastic equations for diffusions in a bounded region.
\newblock {\em Theory Probab. Appl.}, 6:264--274.

\bibitem[Slaby, 2010]{S1}
Slaby, M. (2010).
\newblock Explicit representation of the {S}korokhod map with time dependent
  boundaries.
\newblock {\em Probability and Mathematical Statistics}, 30:29--60.

\bibitem[Sun et~al., 2023]{SWW}
Sun, D., Wu, J., and Wu, P. (2023).
\newblock On distribution dependent stochastic differential equations driven by
  ${G}$-{B}rownian motion.
\newblock {\em arXiv:2302.12539v1}.

\bibitem[Sznitman, 1991]{sznitman1991topics}
Sznitman, A.-S. (1991).
\newblock Topics in propagation of chaos.
\newblock {\em Lecture notes in mathematics}, pages 165--251.

\bibitem[Tanaka, 1979]{T}
Tanaka, H. (1979).
\newblock Stochastic differential equations with reflecting boundary condition
  in convex regions.
\newblock {\em Hiroshima Math. J.}, 9:163--177.

\end{thebibliography}

\end{document}